\documentclass[11pt,english]{amsart}
\usepackage[T1]{fontenc}
\usepackage[latin9]{inputenc}
\usepackage[letterpaper]{geometry}
\usepackage{babel}
\usepackage{amsthm}
\usepackage{amstext}
\usepackage{amssymb}
\usepackage[unicode=true,pdfusetitle,
 bookmarks=true,bookmarksnumbered=false,bookmarksopen=false,
 breaklinks=false,pdfborder={0 0 1},backref=false,colorlinks=false]
 {hyperref}
\usepackage{breakurl}

\makeatletter
\numberwithin{equation}{section}
\numberwithin{figure}{section}

\theoremstyle{theorem}
\newtheorem{thm}{Theorem}
\newtheorem{prop}[thm]{Proposition}
\newtheorem{lem}[thm]{Lemma}
\theoremstyle{definition}
\newtheorem{remark}[thm]{Remark}

\usepackage{tikz}
\tikzstyle{V}=[draw, fill =black, circle, inner sep=0pt, minimum size=2pt]

\makeatother

\begin{document}

\title{Centralizers of the infinite symmetric group}
\author{Zajj Daugherty and Peter Herbrich}
\address{Department of Mathematics, 6188 Kemeny Hall, Dartmouth College, Hanover, NH, 03755.} 
\email{zajj.b.daugherty@dartmouth.edu, peter.herbrich@dartmouth.edu}
\thanks{Author Z.\ Daugherty is partially supported by the National Science Foundation grant DMS-1162010} 
\date{\today}
\maketitle
\tableofcontents{}

\global\long\def\NC{\mathrm{NCSym}}
 \global\long\def\Norm#1{\Vert#1\Vert}
 \global\long\def\NormBig#1{\Bigl\Vert#1\Bigr\Vert}
 \global\long\def\Normbigg#1{\biggl\Vert#1\biggr\Vert}
 \global\long\def\AbsoluteValue#1{|#1|}
 \global\long\def\AbsoluteValueBig#1{\Bigl|#1\Bigr|}
 \global\long\def\AbsoluteValuebigg#1{\biggl|#1\biggr|}
 \global\long\def\BoundedOperators{\mathcal{B}}
 \global\long\def\cA{\mathcal{A}}
 \global\long\def\cB{\mathcal{B}}
 \global\long\def\cC{\mathcal{C}}
 \global\long\def\cD{\mathcal{D}}
 \global\long\def\cE{\mathcal{E}}
 \global\long\def\cF{\mathcal{F}}
 \global\long\def\cG{\mathcal{G}}
 \global\long\def\cH{\mathcal{H}}
 \global\long\def\cI{\mathcal{I}}
 \global\long\def\cJ{\mathcal{J}}
 \global\long\def\cK{\mathcal{K}}
 \global\long\def\cL{\mathcal{L}}
 \global\long\def\cM{\mathcal{M}}
 \global\long\def\cN{\mathcal{N}}
 \global\long\def\cO{\mathcal{O}}
 \global\long\def\cP{\mathcal{P}}
 \global\long\def\cQ{\mathcal{Q}}
 \global\long\def\cR{\mathcal{R}}
 \global\long\def\cS{\mathcal{S}}
 \global\long\def\cT{\mathcal{T}}
 \global\long\def\cU{\mathcal{U}}
 \global\long\def\cV{\mathcal{V}}
 \global\long\def\cW{\mathcal{W}}
 \global\long\def\cX{\mathcal{X}}
 \global\long\def\cY{\mathcal{Y}}
 \global\long\def\cZ{\mathcal{Z}}
 \global\long\def\AA{\mathbb{A}}
 \global\long\def\BB{\mathbb{B}}
 \global\long\def\CC{\mathbb{C}}
 \global\long\def\DD{\mathbb{D}}
 \global\long\def\EE{\mathbb{E}}
 \global\long\def\FF{\mathbb{F}}
 \global\long\def\GG{\mathbb{G}}
 \global\long\def\HH{\mathbb{H}}
 \global\long\def\II{\mathbb{I}}
 \global\long\def\JJ{\mathbb{J}}
 \global\long\def\KK{\mathbb{K}}
 \global\long\def\LL{\mathbb{L}}
 \global\long\def\MM{\mathbb{M}}
 \global\long\def\NN{\mathbb{N}}
 \global\long\def\OO{\mathbb{O}}
 \global\long\def\PP{\mathbb{P}}
 \global\long\def\QQ{\mathbb{Q}}
 \global\long\def\RR{\mathbb{R}}
 \global\long\def\SS{\mathbb{S}}
 \global\long\def\TT{\mathbb{T}}
 \global\long\def\UU{\mathbb{U}}
 \global\long\def\VV{\mathbb{V}}
 \global\long\def\WW{\mathbb{W}}
 \global\long\def\XX{\mathbb{X}}
 \global\long\def\YY{\mathbb{Y}}
 \global\long\def\ZZ{\mathbb{Z}}
 \global\long\def\fa{\mathfrak{a}}
 \global\long\def\fb{\mathfrak{b}}
 \global\long\def\fc{\mathfrak{c}}
 \global\long\def\fd{\mathfrak{d}}
 \global\long\def\fe{\mathfrak{e}}
 \global\long\def\ff{\mathfrak{f}}
 \global\long\def\fg{\mathfrak{g}}
 \global\long\def\fh{\mathfrak{h}}
 \global\long\def\fj{\mathfrak{j}}
 \global\long\def\fk{\mathfrak{k}}
 \global\long\def\fl{\mathfrak{l}}
 \global\long\def\fm{\mathfrak{m}}
 \global\long\def\fn{\mathfrak{n}}
 \global\long\def\fo{\mathfrak{o}}
 \global\long\def\fp{\mathfrak{p}}
 \global\long\def\fq{\mathfrak{q}}
 \global\long\def\fr{\mathfrak{r}}
 \global\long\def\fs{\mathfrak{s}}
 \global\long\def\ft{\mathfrak{t}}
 \global\long\def\fu{\mathfrak{u}}
 \global\long\def\fv{\mathfrak{v}}
 \global\long\def\fw{\mathfrak{w}}
 \global\long\def\fx{\mathfrak{x}}
 \global\long\def\fy{\mathfrak{y}}
 \global\long\def\fz{\mathfrak{z}}
 \global\long\def\fgl{\mathfrak{gl}}
 \global\long\def\fsl{\mathfrak{sl}}
 \global\long\def\fso{\mathfrak{so}}
 \global\long\def\fsp{\mathfrak{sp}}
 \global\long\def\GL{\mathrm{GL}}
 \global\long\def\SL{\mathrm{SL}}
 \global\long\def\SO{\mathrm{SO}}
 \global\long\def\Sp{\mathrm{Sp}}
 \global\long\def\ad{\mathrm{ad}}
 \global\long\def\Aut{\mathrm{Aut}}
 \global\long\def\dim{\mathrm{dim}}
 \global\long\def\End{\mathrm{End}}
 \global\long\def\ev{\mathrm{ev}}
 \global\long\def\half{\hbox{$\frac{1}{2}$}}
 \global\long\def\Hom{\mathrm{Hom}}
 \global\long\def\id{\mathrm{id}}
 \global\long\def\img{\mathrm{img}}
 \global\long\def\Ind{\mathrm{Ind}}
 \global\long\def\ker{\mathrm{ker}}
 \global\long\def\normeq{\unlhd}
 \global\long\def\qtr{\mathrm{qtr}}
 \global\long\def\tr{\mathrm{tr}}
 \global\long\def\Tr{\mathrm{Tr}}
 \global\long\def\vep{\varepsilon}
 \global\long\def\ii{\mathbf{i}}
 \global\long\def\jj{\mathbf{j}}
 \global\long\def\kk{\mathbf{k}}
 \global\long\def\uu{\mathbf{u}}
 \global\long\def\vv{\mathbf{v}}
 \global\long\def\ww{\mathbf{w}}
 \global\long\def\NOTE#1{{\color{blue} #1}}

\begin{abstract}
We review and introduce several approaches to the study of centralizer algebras of the infinite symmetric group $S_\infty$. Our study is led by the double commutant relationships between finite symmetric groups and partition algebras; each approach produces a centralizer algebra that is contained in a partition algebra. Our goal is to incorporate invariants of $S_\infty$, which ties our work to the study of symmetric functions in non-commuting variables. We resultantly explore sequence spaces as permutation modules, which yields families of non-unitary representations of $S_\infty$.

\medskip

\end{abstract}

\section{Introduction~\label{sec:intro}}

\thispagestyle{empty}

Classical Schur-Weyl duality relates the representation theory of
the general linear group $\GL_{n}(\CC)$ and the finite symmetric
group $S_{k}$ via their commuting actions on a common vector space.
Namely, if $V\cong\CC^{n}$, then the diagonal action of $\GL(V)\cong\GL_{n}(\CC)$
on the algebraic $k$-fold tensor product $V^{\otimes k}$ fully centralizes
the permutation action of $S_{k}$ on the tensor factors: 
\begin{equation}\label{eq:SWduality}
\CC S_{k}\cong\End_{\GL(V)}(V^{\otimes k})\qquad\textnormal{when }n=\dim(V)\geq k.
\end{equation}
In\emph{~}\cite{TsilevichVershik2009}, Tsilevich and Vershik extend
this setting to study an infinite symmetric group $S_{\infty}$ by
creating an infinite tensor power $V^{\otimes\infty}$ on which $\GL(V)$
and $S_{\infty}$ share commuting actions. In the following, we study
the action of $S_{\infty}$ on finite tensor powers by shifting the
role of the symmetric group and generalizing $V$ instead.

The group $\GL_{n}(\CC)$ naturally contains $S_{n}$ as the set of $n\times n$
permutation matrices. By replacing $\GL_{n}(\CC)$ with its subgroup $S_{n}$,
and calculating the centralizer of the diagonal action of $S_{n}$
on $V^{\otimes k}$, where $V$ has now become a permutation module,
one acquires the \emph{partition algebra} $P_{k}(n)$. 
The partition algebras arose independently in the work of Martin \cite{Martin91, Martin94, Martin96, Martin00} and Jones \cite{Jones94} as  generalizations of the Temperley-Lieb algebras and the Potts model in statistical mechanics. Their work modernly serves to coalesce the study of several \emph{tensor power centralizer algebras}, including the group algebras of the finite symmetric groups, the Temperley-Lieb algebras, the Brauer algebras, and the algebras of uniform block permutations; all of these are examples of \emph{diagram algebras}, which we discuss in Section~\ref{sec:diagram_algebras}.

Sam and Snowden~\cite{SamSnowden2013} provide a first approach for treating $S_\infty$ as a finite tensor power centralizer algebra. Their work arrives amongst a recently rejuvenated effort to understand representation-theoretic stability of chains of groups, most salient being the chain of finite symmetric groups $S_1 \hookrightarrow S_2 \hookrightarrow S_3 \hookrightarrow \cdots$. Bowman, De Visscher, and Orellana \cite{BDVO12} use the representation-theoretic duality between $S_n$ and $P_k(n)$ to study stability in \emph{Kronecker coefficients}---decomposition numbers for symmetric group representations---when $n$ is large relative to $k$. Church, Ellenberg, and Farb \cite{CEF12} use category-theoretic methods to create corresponding  chains of modules, each of whose structure is tied together into a single \emph{FI-module}. Sam and Snowden's approach is related, but makes the additional connection back to Schur-Weyl duality, and treats many other examples of groups. The case of most interest here is their consideration of the action of $S_{\infty}$ on a countable-dimensional vector space $V\cong\CC^{(\NN)}$, as reviewed in Section~\ref{sec:SamSnowden}. This gives rise to the \emph{downwards} and \emph{upwards partition categories}, whose
homogeneous degree $k$ components each form subalgebras of a partition
algebra; we call these subalgebras \emph{bottom-} and \emph{top-propagating partition algebras}, respectively.

Another motivation for a treatment of $S_\infty$ comes from connections between Hopf algebras and symmetric functions. An influential result of Gessel \cite{Gessel84} provides that the Hopf algebra structure of the \emph{Solomon descent algebra}  is in duality with the algebraic structure of the \emph{quasi-symmetric functions}, and vice-versa. Malvenuto and Reutenauer \cite{MalReu95}  revisit this connection, in essence using the classical relationship between $\GL(V)$ and $S_k$ in \eqref{eq:SWduality}; they exploit the Hopf algebra structure of the tensor algebra 
\begin{equation}\label{eq:tensor_algebra}
T(V)=\bigoplus_{k=0}^{\infty}V^{\otimes k}, \qquad \text{ where } V\cong\CC^n,
\end{equation}
as well as its bi-module structure for $\GL(V)$ and $\bigoplus_{k \geq 0} \CC S_k$, which restricts to the work of Gessel. Aguiar and Orellana generalize \cite{MalReu95} in \cite{AguiarOrellana2008}, drawing on the centralizer relationship between the complex reflection groups $C_r \wr S_n$, in the role of $\GL_n(\CC)$, and the subalgebra $U_k$ of the partition algebra spanned by the \emph{uniform block permutations}. The result is a graded Hopf algebra $U = \bigoplus_{k \geq 0} U_k$, which analogously contains the Hopf algebra of \emph{symmetric functions in non-commuting variables} $\NC$. However, there is a fragility in the centralizer relationships between $C_r \wr S_n$ and $U_k$, necessitating $r$ and $n$ be kept large relative to $k$; a study of $U$ all at once thus suggests the presence of $S_\infty$ at work in the background.

A first guess of how to illuminate the role of $S_\infty$ in the work of \cite{AguiarOrellana2008} might be to let $S_\infty$ act on $\CC^{(\NN)}$ as in \cite{SamSnowden2013} discussed above. However, this particular generalization is missing the $S_\infty$-invariant
structure crucial to the application to symmetric functions. To see how, start from the finite-dimensional case, were $V$
has basis $\{v_{1},\dots,v_{n}\}$ over $\CC$. Then $V^{\otimes k}$
can be canonically identified with the space of homogeneous polynomials
of degree $k$ in non-commuting variables $v_{1},\dots,v_{n}$; the tensor algebra in \eqref{eq:tensor_algebra} is isomorphic to the full polynomial ring.
Then, the symmetric functions
$\NC$ correspond to elements of $T(V)$ that are invariant under
the action of $S_{n}$ by permuting $v_{1},\dots,v_{n}$;
in each homogeneous degree $k$, symmetric functions correspond to the invariant elements of
$V^{\otimes k}$. For example, 
\[
\left(V\right)^{S_{n}}=\CC\Bigl\{\sum_{i=1}^{n}v_{i}\Bigr\} \qquad\text{ and } \qquad
\left(V\otimes V\right)^{S_{n}}=\CC\Bigl\{\sum_{i=1}^{n}v_{i}\otimes v_{i},\;\sum_{i,j=1}^{n}v_{i}\otimes v_{j}\Bigr\}.
\]
For a complete treatment of all symmetric functions, one passes to the
case of countably many variables, where $S_{\infty}$ acts on functions
in non-commuting variables $\{v_{i}\}_{i\in\NN}=\{v_{1},v_{2},\ldots\}$.
No non-trivial $S_\infty$-invariants exist in the tensor space $V^{\otimes k}$ if $V\cong\CC^{(\NN)}$
as in\emph{~}\cite{SamSnowden2013}; $\CC^{(\NN)}$ contains only finite linear combinations of its basis elements.

In Section~\ref{sec:sequence_spaces},
we present two other choices of vector space $V$ in order to capture
the invariant structure discussed above. In each case, a countable set  $\{v_1, v_2, \dots\}$ is contained in $V$, and endomorphisms considered are determined by their images on that set. In Section \ref{subsec:p_Power_summable_sequences}, we let $V$ be a Banach space of $p$-power summable sequences $\ell^p$, and impart an action of $S_\infty$. A metric is chosen so that $V$ contains the desired $S_\infty$-invariants, and we show that all desired invariants appear in subsequent tensor powers. Theorem \ref{thm:Centralizer_in_Lp_case} then states that the centralizer of the action of $S_\infty$ on $\overline{V^{\otimes k}}$ inside the set of bounded maps is the same algebra $U_k$ used in \cite{AguiarOrellana2008}. 
In Section \ref{subsec:l_infinity}, we instead let $V$ be the Banach space of bounded sequences $\ell^\infty$,  and describe suitable analogs to $k$-fold tensor powers and bounded maps. Again, each space contains all desired $S_\infty$-invariants. Theorem \ref{thm:Centralizer_in_l_infty} states that the centralizer algebra is again a subalgebra of the partition algebra, this time isomorphic to the top-propagating partition algebra arising in \cite{SamSnowden2013}.

It is notable that the inclusion of non-trivial $S_\infty$-invariants into our permutation modules comes at a cost in return. Namely, the representations studied in Section~\ref{sec:sequence_spaces} are not unitary. The non-unitary representation
theory of wild groups, of which $S_{\infty}$ is an example, is largely
intractable; we refer to\emph{~}\cite{Okounkov1997} for a survey
of the representation theory of $S_{\infty}$ and to\emph{~}\cite{Kirillov1994}
for an exposition on tame and wild groups. By including the desired
$S_{\infty}$-invariants, we acquire representations which are reducible,
by design, but not fully decomposable. This is reflected in the fact
that the centralizer algebras of the actions of $S_{\infty}$ are small in some sense; in particular, the double commutant property present in classical Schur-Weyl duality 
does not hold, as discussed in Remark~\ref{rk:S_infty_not_fully_decomposable}. However, a further study of these non-unitary representations may still be made more manageable with leverage provided by the centralizer algebras calculated here.

\smallskip\noindent
\textbf{Acknowledgements:}  The first author would like to thank Aaron Lauve for bringing to her attention the link between diagram Hopf algebras and symmetric functions in non-commuting variables, thus inspiring the question of how to place such functions into a centralizer algebra framework.

\section{Diagram algebras \label{sec:diagram_algebras}}

A \emph{set partition} of a set $S$ is a set of pairwise disjoint
subsets of $S$, called \emph{blocks}, whose union is $S$. Fix $k\in\NN=\{1,2,\ldots\}$,
and denote 
\[
[k]=\{1,\dots,k\}\qquad\textnormal{and}\qquad[k']=\{1',\dots,k'\},
\]
so that $[k]\cup[k']=\{1,\dots,k,1',\dots,k'\}$ is formally a set
with $2k$ elements. To each set partition of $[k]\cup[k']$, we associate
an equivalence class of graphs, called a ($k$-)\emph{diagram}, as
follows. Consider the set of graphs with vertices $[k]\cup[k']$,
and let two graphs be equivalent if they have the same connected components.
To each diagram $d$ associate the set partition of $[k]\cup[k']$
determined by the connected components of any of its representatives.
For example, \def\UNIT{.75}
\begin{equation}
\label{diagram-equivalence}
\begin{matrix}
 \begin{tikzpicture}[scale=\UNIT]
  \foreach \x in {1,...,4} {
   \node[above] at (\x,1) {\tiny \x}; 
   \node[below] at (\x,0) {\tiny \x'};}
  \foreach \y in {0,1}{
   \foreach \x in {1, ..., 4}{
    \node[V] at (\x,\y) {};}}
  \draw (1,1)--(1,0)--(2,1) (2,0)--(3,0)--(4,0)--(4,1);
 \end{tikzpicture}
\end{matrix}
\qquad \text{ and } \qquad
\begin{matrix}
 \begin{tikzpicture}[scale=\UNIT]
  \foreach \x in {1,...,4}{
   \node[above] at (\x,1) {\tiny \x};
   \node[below] at (\x,0) {\tiny \x'};}
  \foreach \y in {0,1}{
   \foreach \x in {1, ..., 4}{
    \node[V] at (\x,\y) {};}}
  \draw (1,0)--(1,1)--(2,1) (4,1)--(3,0)--(4,0)--(4,1)--(2,0);
 \end{tikzpicture}
\end{matrix}
\end{equation}

\noindent are equivalent, and both represent diagrams for the set
partition $\{\{1,2,1'\},\{3\},\{4,2',3',4'\}\}$. Let $D_{k}$ be
the set of $k$-diagrams.

Let $\CC(x)$ be the field of rational functions with complex coefficients
in an indeterminate $x$. The product $d_{1}*d_{2}$ of two diagrams
$d_{1}$ and $d_{2}$ is defined as the concatenation of $d_{1}$
above $d_{2}$, where one identifies the bottom vertices of $d_{1}$
with the top vertices of $d_{2}$, and removes any components consisting
only of middle vertices. This defines the \emph{partition monoid},
which can be extended to an algebra as follows. If there are $m$
middle components in the concatenation of $d_{1}$ and $d_{2}$ as
before, let $d_{1}d_{2}=x^{m}d_{1}*d_{2}$, and extend linearly. For
example, \def\UNIT{.75}
$${
\begin{matrix}
 \begin{tikzpicture}[scale=\UNIT]
  \foreach \y in {0,1}{
   \foreach \x in {1, ..., 4}{
    \node[V] at (\x,\y) {};}}
  \draw (1,1)--(2,1) (1,0)--(2,0) (3,0)--(4,0)--(4,1);
 \end{tikzpicture}
\end{matrix}
\quad \cdot \quad
\begin{matrix}
 \begin{tikzpicture}[scale=\UNIT]
  \foreach \y in {0,1}{
   \foreach \x in {1, ..., 4}{
   \node[V] at (\x,\y) {};}}
	\draw  (3,1)-- (1,0) (4,1)--(4,0)--(3,0)--(2,0);
 \end{tikzpicture}
\end{matrix}
\quad = \quad x \cdot ~
\begin{matrix}
 \begin{tikzpicture}[scale=\UNIT] 	
  \foreach \y in {0,1}{
   \foreach \x in {1, ..., 4}{
    \node[V] at (\x,\y) {};}} 		
  \draw (1,1)--(2,1) (4,1)--(4,0)--(3,0)--(2,0)--(1,0);
 \end{tikzpicture}
\end{matrix}
}$$

\noindent This product is associative and independent of the graphs
chosen to represent the partition diagrams.

The \emph{partition algebra} $P_{k}(x)$ is the span over $\CC(x)$
of the set $D_{k}$ of $k$-diagrams equipped with this product, where
$P_{0}(x)=\CC(x)$. The vector space $P_{k}(x)$ is an associative
algebra with identity given by the diagram corresponding to $\{\{1,1'\},\dots,\{k,k'\}\}$.
The dimension of $P_{k}(x)$ is the number of set partitions of $2k$
elements, that is, the \emph{Bell number} $B(2k)$.

Each partition algebra contains many important subalgebras, including
group algebras of symmetric groups, Brauer algebras, and Temperley-Lieb
algebras; see\emph{~}\cite[Section 1]{HalversonRam2005}.
Three specific subalgebras of $P_{k}(x)$ will be of interest in later sections: $U_k$, $TP_k$, and $BP_k$. The algebra
$U_{k}$ of \emph{uniform block permutations} is spanned by
the set of diagrams $d$ satisfying 
\[
\AbsoluteValue{B\cap[k]}=\AbsoluteValue{B\cap[k']}\qquad\textnormal{for every }B\in d.
\]
For example, \def\UNIT{.75}
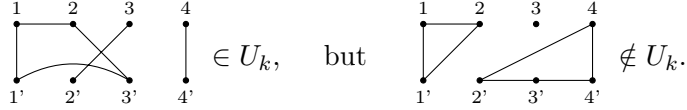
$${
\begin{matrix}
 \begin{tikzpicture}[scale=\UNIT]
  \foreach \x in {1,...,4}{
   \node[above] at (\x,1) {\tiny \x};
   \node[below] at (\x,0) {\tiny \x'};}
  \foreach \y in {0,1}{\foreach \x in {1, ..., 4}{
  \node[V] at (\x,\y) {};}}
  \draw (1,1)--(1,0) to [bend left] (3,0)--(2,1)--(1,1) (3,1)--(2,0) (4,0)--(4,1); 	
 \end{tikzpicture}
\end{matrix}
\in U_k, \quad \text{ but }\quad 
\begin{matrix}
 \begin{tikzpicture}[scale=\UNIT]
  \foreach \x in {1,...,4}{
   \node[above] at (\x,1) {\tiny \x};
   \node[below] at (\x,0) {\tiny \x'};}
  \foreach \y in {0,1}{
   \foreach \x in {1, ..., 4}{
    \node[V] at (\x,\y) {};}}
  \draw (1,1)--(1,0)--(2,1)--(1,1) (2,0)--(3,0)--(4,0)--(4,1)--(2,0); 	
 \end{tikzpicture}
\end{matrix}
\notin U_k. 
}$$
\noindent More on the role of $U_{k}$ can be found in Theorem \ref{thm:Centralizer_in_Lp_case} and Remark~\ref{rk:Uk_as_module_for_S_infty}.

We say a block is \emph{propagating} if it contains vertices in both the top and bottom of a diagram. The \emph{top-propagating partition algebra} $TP_k$ is the subalgebra spanned by diagrams for which all blocks containing top vertices are propagating, 
\begin{equation}\label{eq:TP}
TP_k = \CC\left\{ d\in D_{k}\,\bigl|\,\textnormal{for every }B\in d\colon B\cap[k]\neq B\right\}.
\end{equation}
Similarly, the \emph{bottom-propagating partition algebra} $BP_k$ is the subalgebra spanned by diagrams which have no blocks isolated to the bottom. 
For example, \def\UNIT{.75}
$$
\begin{matrix}
 \begin{tikzpicture}[xscale=\UNIT, yscale=-\UNIT]
  \foreach \x in {1,...,4}{
   \node[below] at (\x,1) {\tiny \x'};
   \node[above] at (\x,0) {\tiny \x};}
  \foreach \y in {0,1}{
   \foreach \x in {1, ..., 4}{
    \node[V] at (\x,\y) {};}}
  \draw (1,1)--(1,0) (2,1)--(3,1) (2,0)--(3,0)--(4,0)--(4,1)--(2,0);
 \end{tikzpicture}
\end{matrix}
\in TP_{k} \quad \text{ and }\quad
\begin{matrix}
 \begin{tikzpicture}[scale=\UNIT]
  \foreach \x in {1,...,4}{
   \node[above] at (\x,1) {\tiny \x};
   \node[below] at (\x,0) {\tiny \x'};}
  \foreach \y in {0,1}{
   \foreach \x in {1, ..., 4}{
    \node[V] at (\x,\y) {};}}
  \draw (1,1)--(1,0) (2,1)--(3,1) (2,0)--(3,0)--(4,0)--(4,1)--(2,0); 	
 \end{tikzpicture}
\end{matrix}
\in BP_k.
$$ 
These algebras appear in Sections \ref{sec:SamSnowden} and \ref{subsec:l_infinity}, respectively. Note that for each of $U_k$, $TP_k$, and $BP_k$, concatenation of diagrams never results in middle components, and therefore none of these algebras are dependent on the parameter $x$.

\section{Vector spaces of finite or countable dimension as permutation modules
\label{sec:vector_space}}

\subsection{Finite symmetric group $S_{n}$ and its action on $(\CC^{n})^{\otimes k}$}

Let $V$ denote the $n$-dimensional permutation representation of
the symmetric group $S_{n}$. That is, let $V$ have basis $\{v_{1},\dots,v_{n}\}$,
on which $S_{n}$ acts by 
\[
\sigma\cdot v_{i}=v_{\sigma(i)}\qquad\textnormal{for }\sigma\in S_{n}.
\]
For $\ii=(i_{1},\dots,i_{k})\in[n]^{k}$, a $k$-tuple of integers
in $\{1,\dots,n\}$, and $\sigma\in S_{n}$, we let 
\[
v_{\ii}=v_{i_{1}}\otimes\dots\otimes v_{i_{k}}\in V^{\otimes k}\qquad\textnormal{and}\qquad\sigma(\ii)=(\sigma(i_{1}),\dots,\sigma(i_{k})).
\]
Let $S_{n}$ act diagonally on the basis $\{v_{\ii}\}_{\ii\in[n]^{k}}$
of $V^{\otimes k}$, that is, 
\[
\sigma\cdot v_{\ii}=v_{\sigma(\ii)},
\]
and extend this action linearly to $V^{\otimes k}$. Thus, $V^{\otimes k}$
becomes a module for $S_{n}$.

As in Section~\ref{sec:diagram_algebras}, arrange the vertices of
a $k$-diagram reading $1,\ldots,k$ from left to right in the top
row and $1',\ldots,k'$ from left to right in the bottom row. For
each $k$-diagram $d$ and each $2k$-tuple of integers $i_{1}\ldots,i_{k},i_{1'},\ldots,i_{k'}\in[n]$,
we define 
\begin{equation}
d_{(i_{1'},\ldots,i_{k'})}^{(i_{1},\ldots,i_{k})}=\begin{cases}
1 & \textnormal{ if }i_{\ell}=i_{m}\textnormal{ whenever vertices }\ell\textnormal{ and }m\textnormal{ are connected in }d,\\
0 & \textnormal{otherwise}.
\end{cases}\label{eq:diagram_matrix_entries}
\end{equation}
For example, \def\UNIT{.5}
\def\TRIANGLE{
 \begin{matrix}
  \begin{tikzpicture}[scale=\UNIT]
   \foreach \x in {1,2}{
    \node[V, label=above:{\tiny $\x$}] at (\x, 1) {};
    \node[V, label=below:{\tiny $\x'$}] at (\x, 0){};}
   \draw (1,0)--(1,1)--(2,0)--(1,0); 	
  \end{tikzpicture} 
 \end{matrix}
} 
\[
\Bigg(\TRIANGLE\Bigg)_{(3,5)}^{(3,7)}=0\qquad\textnormal{and}\qquad\Bigg(\TRIANGLE\Bigg)_{(3,3)}^{(3,7)}=\Bigg(\TRIANGLE\Bigg)_{(4,4)}^{(4,4)}=1.
\]
The algebra $P_{k}(n)$ acts on $V^{\otimes k}$; namely, for each
$d\in P_{k}(n)$ and $\ii\in[n]^{k}$, we define 
\[
d\cdot v_{\ii}=\sum_{\jj\in[n]^{k}}d_{\ii}^{\jj}v_{\jj},
\]
and extend linearly. For example, we have the following identities
in the action of $P_{2}(n)$ on $V^{\otimes2}$ { \def\UNIT{.5}
\begin{equation}
\label{actionsofP}
 \begin{array}{c@{~}c}
  \begin{matrix}
   \begin{tikzpicture}[scale=\UNIT]
    \foreach \x in {1,2}{
     \filldraw [black] (\x, 0) circle (2pt);
     \filldraw [black] (\x, 1) circle (2pt);} 	
    \draw (1,0)--(1,1)--(2,1)-- (2,0)--(1,0); 	
   \end{tikzpicture} 
  \end{matrix}
  \cdot (v_i\otimes v_j) = \delta_{ij} v_i\otimes v_i,  & \qquad 
  \begin{matrix} 
   \begin{tikzpicture}[scale=\UNIT]
    \foreach \x in {1,2}{
     \filldraw [black] (\x, 0) circle (2pt); 
     \filldraw [black] (\x, 1) circle (2pt);}
    \draw (1,0)--(2,1) (2,0)--(1,1);
   \end{tikzpicture} 
  \end{matrix}  
  \cdot (v_i\otimes v_j) = v_j \otimes v_i, \\  
  \begin{matrix} 
   \begin{tikzpicture}[scale=\UNIT] 	 
    \foreach \x in {1,2} {
     \filldraw [black] (\x, 0) circle (2pt);
     \filldraw [black] (\x, 1) circle (2pt);}
    \draw (1,1)--(2,1) (2,0)--(1,0); 	
   \end{tikzpicture} 
  \end{matrix}  
  \cdot (v_i\otimes v_j) =  \delta_{ij}\sum\limits_{\ell=1}^n v_\ell \otimes v_\ell,   \qquad &  \text{ and } \qquad    
  \begin{matrix} 
   \begin{tikzpicture}[scale=\UNIT] 	 
    \foreach \x in {1,2} {
     \filldraw [black] (\x, 0) circle (2pt); 
     \filldraw [black] (\x, 1) circle (2pt);} 	
    \draw (1,0)--(2,1); 	
   \end{tikzpicture} 
  \end{matrix}
  \cdot (v_i\otimes v_j) = \left(\sum\limits_{\ell=1}^n v_\ell\right) \otimes v_i. 
 \end{array}
\end{equation}} 

\noindent \begin{thm}[\cite{Jones94}]\label{thm:PkFullCentralizer}
$S_{n}$ and $P_{k}(n)$ generate full centralizers of each other
in $\mbox{End}(V^{\otimes k})$. In fact,
\begin{enumerate}
\item $P_{k}(n)$ generates $\mbox{End}_{S_{n}}(V^{\otimes k})$, and when
$n\geq2k$, we have $P_{k}(n)\cong\mbox{End}_{S_{n}}(V^{\otimes k})$; 
\item $S_{n}$ generates $\mbox{End}_{P_{k}(n)}(V^{\otimes k})$. 
\end{enumerate}
\end{thm}

The main calculation in the proof of part (1) will be used in several
settings later, so let us review. Let $A\in\End(V^{\otimes k})$ be
given by the matrix $(A_{\ii}^{\jj})_{\ii,\jj\in[n]^{k}}$ such that
for each $\ii\in[n]^{k}$, 
\[
A(v_{\ii})=\sum_{\jj\in[n]^{k}}A_{\ii}^{\jj}v_{\jj}.
\]
If $\sigma\in S_{n}$, and $\sigma A=A\sigma$ in $\End(V^{\otimes k})$,
then for each $\ii\in[n]^{k}$, 
\begin{equation}
\sigma A(v_{\ii})=\sum_{\jj\in[n]^{k}}A_{\ii}^{\jj}v_{\sigma(\jj)}\qquad\textnormal{equals}\qquad A\sigma(v_{\ii})=\sum_{\jj\in[n]^{k}}A_{\sigma(\ii)}^{\jj}v_{\jj}=\sum_{\jj\in[n]^{k}}A_{\sigma(\ii)}^{\sigma(\jj)}v_{\sigma(\jj)},\label{eq:finite_commutation}
\end{equation}
since $\sigma$ is a bijection of $[n]^{k}$. So for every $\sigma\in S_{n}$,
we have 
\[
\sigma A=A\sigma\quad\textnormal{if and only if}\quad A_{\sigma(\ii)}^{\sigma(\jj)}=A_{\ii}^{\jj}\quad\textnormal{for every }\ii,\jj\in[n]^{k}.
\]
Thus, $A\in\End_{S_{n}}(V^{\otimes k})$ if and only if the entries
of $(A_{\ii}^{\jj})_{\ii,\jj\in[n]^{k}}$ are uniform on $S_{n}$-orbits,
which exactly describes those linear transformations coming from $P_{k}(n)$.

\subsection{Infinite symmetric group $S_{\infty}$ and its action on $(\CC^{(\NN)})^{\otimes k}$}
\label{sec:SamSnowden} 

Embed $S_{n}\hookrightarrow S_{n+1}$ as the
subgroup which fixes $n+1$. Then, let $S_{\infty}$ be the direct
limit of $\{S_{n}\}_{n\in\NN}$, that is, the permutations of $\NN$
which fix all but finitely many elements. Let $V$ be a countable-dimensional
vector space with basis $\{v_{i}\}_{i\in\NN}$, that is, 
\[
V=\CC\{v_{i}\}_{i\in\NN}=\left\{ \sum_{i\in\NN}a_{i}v_{i}\,\biggl|\, a_{i}=0\textnormal{ for all but finitely many }i\right\} \cong\CC^{(\NN)}.
\]
The same calculation as in~\eqref{eq:finite_commutation} leads to
the same conclusion; namely, $A\in\End_{S_{\infty}}(V^{\otimes k})$
if and only if the entries of its matrix representation $(A_{\ii}^{\jj})_{\ii,\jj\in\NN^{k}}$
with respect to $\{v_{\ii}\}_{\ii\in\NN^{k}}$ are uniform on $S_{\infty}$-orbits.
So $\End_{S_{\infty}}(V^{\otimes k})$ is still spanned by endomorphisms
coming from diagrams in $D_{k}$. However, endomorphisms of $V^{\otimes k}$
have images in $V^{\otimes k}$, that is, for every $\ii\in\NN^{k}$,
the set $\{\jj\in\NN^{k}\mid A_{\ii}^{\jj}\neq0\}$ is finite. So
$\End_{S_{\infty}}(V^{\otimes k}) \cong TP_k,$ the top-propagating partition algebra defined in \eqref{eq:TP}. 

In\emph{~}\cite{SamSnowden2013}, Sam and Snowden study the endomorphisms
of 
\[
T(V)=\bigoplus_{k=0}^{\infty}V^{\otimes k}
\]
generated by \emph{$k,\ell$-diagrams} on $[k]\cup[\ell']$, where $k$ and $\ell$
are not necessarily equal, that have no blocks isolated to the top row. Multiplication between a $k_{1},\ell_{1}$-diagram
and a $k_{2},\ell_{2}$-diagram is defined as concatenation when $\ell_{1}=k_{2}$
and is zero otherwise. No middle components arise in resolving concatenations
because there are no blocks isolated to the top row of any diagram;
so multiplication involves no parameter. Under this multiplication, the  {$k,\ell$-diagrams} for $k, \ell \in \NN$ generate the \emph{upwards
partition algebra} $UP$\footnote{Note that here, we write all actions as left actions, whereas the actions in \cite{SamSnowden2013} are right actions; therefore our use of upwards and downwards is reversed.}. The top-propagating partition algebra is exactly the degree-$k$ homogeneous component of $UP$. 

\section{Sequence spaces as permutation modules}\label{sec:sequence_spaces}

The study of symmetric functions in countably many variables requires
to leave the finite-dimensional realm and instead consider vector
spaces $V$ that contain countable linearly independent subsets $\{v_{i}\}_{i\in\NN}$.
With an eye toward studying $S_{\infty}$ invariants as mentioned
in Section~\ref{sec:intro}, we require $\sum_{i=1}^{\infty}v_{i}$
to be interpretable as an element of $V$. This rules out the possibility
of $\{v_{i}\}_{i\in\NN}$ being a \emph{Hamel basis} of $V$ as in
Section~\ref{sec:SamSnowden}, in which case every vector has a unique
expression as a finite linear combination of the basis vectors $\{v_{i}\}_{i\in\NN}$.
We propose the following approach:

\begin{center}
\begin{minipage}[t]{0.95\columnwidth}%
Choose a vector space $V$ containing a countable linearly independent
subset $\{v_{i}\}_{i\in\NN}$ and a vector $\sum_{i=1}^{\infty}v_{i}$
that is invariant under $\CC S_{\infty}$, which is considered as
a subalgebra of an algebra of endomorphisms of $V$ that are determined
by their images on $\{v_{i}\}_{i\in\NN}$.%
\end{minipage}
\par\end{center}

\subsection{$p$-power summable sequences\label{subsec:p_Power_summable_sequences}}

We recall definitions from Banach space theory that will be used throughout
this section. A sequence $\{w_{\ell}\}_{\ell\in\NN}$ in a normed
vector space $V$ is called a \emph{Cauchy sequence}, if for every
$\varepsilon>0$ there is some $L\in\NN$ such that for all integers
$\ell,m>L$, we have $\Norm{w_{\ell}-w_{m}}<\varepsilon$. A normed
vector space $V$ is called a \emph{Banach space} if every Cauchy
sequence $\{w_{\ell}\}_{\ell\in\NN}$ in $V$ converges to some vector
$w=\lim_{\ell\to\infty}w_{\ell}$ in $V$, meaning $\lim_{\ell\to\infty}\Norm{w-w_{\ell}}=0$.
A sequence $\{v_{i}\}_{i\in\NN}$ in a Banach space $V$ is called
a \emph{Schauder basis} if for every $v\in V$ there exist unique
scalars $\{a_{i}\}_{i\in\NN}$ such that 
\begin{equation}\label{eq:Sc-basis}
v=\sum_{i=1}^{\infty}a_{i}v_{i}=\lim_{\ell\to\infty}\sum_{i=1}^{\ell}a_{i}v_{i}\qquad\textnormal{meaning}\qquad\lim_{\ell\to\infty}\,\Normbigg{v-\sum_{i=1}^{\ell}a_{i}v_{i}}=0;
\end{equation}
qualitatively, a Schauder basis is a linearly independent set such that every element of $V$ can be written uniquely as in \eqref{eq:Sc-basis}. 
The basis $\{v_{i}\}_{i\in\NN}$ is called \emph{unconditional} if
the convergence is always unconditional. If now $A$ is a continuous
endomorphism on $V$, and if $v=\sum_{i=1}^{\infty}a_{i}v_{i}$, then
$\{A(v_{i})\}_{i\in\NN}$ determines $A(v)$ since 
\[
A(v)=A\biggl(\lim_{\ell\to\infty}\sum_{i=1}^{\ell}a_{i}v_{i}\biggr)=\lim_{\ell\to\infty}A\biggl(\sum_{i=1}^{\ell}a_{i}v_{i}\biggr)=\lim_{\ell\to\infty}\sum_{i=1}^{\ell}a_{i}A(v_{i}).
\]
We therefore study the following special case of the above-mentioned
approach:

\begin{center}
\begin{minipage}[t]{0.95\columnwidth}%
Choose a Banach space $V$ with a countable Schauder basis\emph{ }$\{v_{i}\}_{i\in\NN}$
such that $\sum_{i=1}^{\infty}v_{i}$ converges, and study $\CC S_{\infty}$
as a subalgebra of the algebra of continuous endomorphisms of $V$.%
\end{minipage}
\par\end{center}

\noindent To this end, we consider $L^{p}$-spaces of sequences of
the form 
\begin{equation}
V=L^{p}(\NN,\mu)=\left\{ v=(a_{1},a_{2},\ldots)\in\CC^{\NN}\,\biggl|\,\Norm v^{p}=\sum_{i=1}^{\infty}\AbsoluteValue{a_{i}}^{p}\mu_{i}^{p}<\infty\right\} ,\label{eq:Lp_space}
\end{equation}
where $1\leq p<\infty$, and $\mu$ is a weighted counting measure
on $\NN$, which is determined by a sequence $(\mu_{i})_{i\in\NN}$
with $\mu_{i}>0$ for all $i\in\NN$ via $\mu(\{i\})=\mu_{i}^{p}$.
The space $\ell^{\infty}=L^{\infty}(\NN,\mu)$ of bounded sequences
will be dealt with in Section~\ref{subsec:l_infinity}. The normed
vector space~\eqref{eq:Lp_space} is a Banach space that has unconditional
Schauder bases such as $\{v_{i}\}_{i\in\NN}$ given by $v_{i}=(\delta_{ij})_{j\in\NN}$,
that is, $v_{1}=(1,0,0,\ldots)$, $v_{2}=(0,1,0\ldots)$, etc. In
particular, $v=(a_{1},a_{2},\ldots)\in V$ if and only if $\sum_{i\in\NN}a_{i}v_{i}$
converges unconditionally to $v$ in $V$. This allows to introduce
the notation $\sum_{i\in\NN}a_{i}v_{i}$ for arbitrary $(a_{1},a_{2},\ldots)\in\CC^{\NN}$
so that 
\[
V=\left\{ v=\sum_{i\in\NN}a_{i}v_{i}\in\CC^{\NN}\,\biggl|\,\Norm v^{p}=\sum_{i=1}^{\infty}\AbsoluteValue{a_{i}}^{p}\mu_{i}^{p}<\infty\right\} .
\]
In order to ensure that $\sum_{i\in\NN}v_{i}=(1,1,1,\ldots)\in V$,
we henceforth require that $(\mu_{i})_{i\in\NN}\in\ell^{p}$, that
is, 
\[
\sum_{i=1}^{\infty}\mu_{i}^{p}<\infty.
\]
We turn to the algebraic $k$-fold tensor product $V^{\otimes k}=(L^{p}(\NN,\mu))^{\otimes k}$,
and note that it carries a canonical cross norm so that its completion
$\overline{V^{\otimes k}}$ is isomorphic to $V$~\cite[Chapter 7]{DefantFloret1993};
namely, 
\[
\overline{V^{\otimes k}}=L^{p}(\NN^{k},\mu^{k})=\Biggl\{ v=\sum_{\ii\in\NN^{k}}a_{\ii}v_{\ii}\in\CC^{\NN^{k}}\,\biggl|\,\Norm v^{p}=\sum_{\ii\in\NN^{k}}\AbsoluteValue{a_{\ii}}^{p}\mu_{\ii}^{p}<\infty\Biggr\},
\]
where $\sum_{\ii\in\NN^{k}}a_{\ii}v_{\ii}$ represents the function
$v\colon\NN^{k}\to\CC$ given by $v(\ii)=a_{\ii}$, and for $\ii=(i_{1},\dots,i_{k})\in\NN^{k}$,
\[
v_{\ii}=v_{i_{1}}\otimes\cdots\otimes v_{i_{k}}\qquad\textnormal{and}\qquad\mu_{\ii}=\prod_{\ell=1}^{k}\mu_{i_{\ell}}=\Norm{v_{\ii}}.
\]
In particular, $\{v_{\ii}\}_{\ii\in\NN^{k}}$ is an unconditional
Schauder basis of $\overline{V^{\otimes k}}$, and $V^{\otimes k}$
can be identified with the dense subset of linear combinations of
vectors of the form 
\[
\sum_{\ii=(i_{1},\dots,i_{k})\in\NN^{k}}\Biggl(\prod_{\ell=1}^{k}a_{\ell i_{\ell}}\Biggr)v_{\ii}=\Biggl(\sum_{i_{1}\in\NN}a_{1i_{1}}v_{i_{1}}\Biggr)\otimes\ldots\otimes\Biggl(\sum_{i_{k}\in\NN}a_{ki_{k}}v_{i_{k}}\Biggr)\textnormal{ with }\sup_{1\leq\ell\leq k}\,\sum_{i_{\ell}=1}^{\infty}\AbsoluteValue{a_{\ell i_{\ell}}}^{p}\mu_{i_{\ell}}^{p}<\infty.
\]
We point out that there are several so-called \emph{reasonable cross
norms} on tensor products of Banach spaces, ranging from the injective
to the projective one as introduced by Grothendieck, but refer to~\cite{DefantFloret1993}
for a detailed exposition.

A linear operator on a Banach space, say $A\colon\overline{V^{\otimes k}}\to\overline{V^{\otimes k}}$,
is continuous if and only if it is \emph{bounded}, meaning it maps
bounded sets to bounded sets, which happens precisely if it has finite
\emph{operator norm}, 
\[
\Norm A=\sup_{v\in\overline{V^{\otimes k}}\colon\Norm v\leq1}\Norm{A(v)}.
\]
Moreover, the set of bounded operators 
\[
\mathcal{B}(\overline{V^{\otimes k}})=\left\{ A\in\End(\overline{V^{\otimes k}})\,\Bigl|\,\Norm A<\infty\right\} 
\]
is a Banach space with respect to the operator norm.

We let $S_{\infty}$ act diagonally on the basis $\{v_{\ii}\}_{\ii\in\NN^{k}}$
of $\overline{V^{\otimes k}}$ so that for $\sigma\in S_{\infty}$
and $\ii=(i_{1},\dots,i_{k})\in\NN^{k}$,
\[
\sigma\cdot v_{\ii}=v_{\sigma(\ii)},\qquad\text{ where}\qquad\sigma(\ii)=(\sigma(i_{1}),\dots,\sigma(i_{k})),
\]
and we extend linearly. In other words, if $\sigma\in S_{\infty}$
and $v=\sum_{\ii\in\NN^{k}}a_{\ii}v_{\ii}\in\overline{V^{\otimes k}}$,
then 
\[
\sigma\cdot v=\sum_{\ii\in\NN^{k}}a_{\ii}v_{\sigma(\ii)}=\sum_{\ii\in\NN^{k}}a_{\sigma^{-1}(\ii)}v_{\ii}\in\overline{V^{\otimes k}}.
\]
Each $\sigma\in S_{\infty}$ gives rise to a continuous endomorphism
since it fixes all but finitely many of the vectors $\{v_{\ii}\}_{\ii\in\NN^{k}}$.
Thus, $\CC S_{\infty}$ can be regarded as a subalgebra of $\mathcal{B}(\overline{V^{\otimes k}})$.
In view of the connection to symmetric polynomials, we define for
each set partition $\pi$ of $[k]$, 
\[
\pi_{(i_{1},\ldots,i_{k})}=\begin{cases}
1 & \textnormal{if }i_{\ell}=i_{m}\textnormal{ whenever }\ell\textnormal{ and }m\textnormal{ are in the same block of }\pi,\\
0 & \text{ otherwise,}
\end{cases}
\]
and let 
\begin{equation}
m_{\pi}=\sum_{\ii\in\NN^{k}}\pi_{\ii}v_{\ii}\in\CC^{\NN^{k}}.\label{eq:Monomial_basis_for_S_infty_invariant_subspace}
\end{equation}
For example, 
\[
m_{\{\{1,2,3\}\}}=\sum_{i\in\NN}v_{i}\otimes v_{i}\otimes v_{i}\qquad\text{ and }\qquad m_{\{\{1,3\},\{2\}\}}=\sum_{i,j\in\NN}v_{i}\otimes v_{j}\otimes v_{i}.
\]
These elements correspond to the so-called \emph{monomial symmetric functions} of degree $k$ in $\NC$, which form a basis for the homogeneous symmetric functions of degree $k$ in non-commuting variables. 
The following proposition reveals that the finitely many functions
defined by~\eqref{eq:Monomial_basis_for_S_infty_invariant_subspace}
are elements of $\overline{V^{\otimes k}}$, and that they yield a
basis of the subspace $(\overline{V^{\otimes k}})^{S_{\infty}}$ of
$S_{\infty}$-invariant vectors. In particular, $\overline{V^{\otimes k}}$
encapsulates symmetric polynomials of degree $k$ in countably
many non-commuting variables.

\begin{prop}\label{prop:S_infty_invariant_subspace} The vector space
$(\overline{V^{\otimes k}})^{S_{\infty}}$ has $\{m_{\pi}\mid\pi\textnormal{ a set partition of }[k]\}$
as a basis.\end{prop}

\begin{proof} We first show that the number of orbits of the diagonal
action of $S_{\infty}$ on $\NN^{k}$ equals the Bell number $B(k)$,
that is, the number of set partitions of $[k]$. To this end, we associate
to each $\ii=(i_{1},\dots,i_{k})\in\NN^{k}$ a set partition $\pi(\ii)$
of $[k]$ via the equivalence relation $\ell\sim m$ if and only if
$i_{\ell}=i_{m}$. For example, 
\[
\pi((1,2,2,3,1))=\pi((3,1,1,2,3))=\pi((4,1,1,2,4))=\{\{1,5\},\{2,3\},\{4\}\}.
\]
Note that $\pi(\ii)=\pi(\jj)$ if and only if $\ii=\sigma(\jj)$
for some $\sigma\in S_{\infty}$, which yields the claimed statement.
Hence, if $v=\sum_{\ii\in\NN^{k}}a_{\ii}v_{\ii}\in\CC^{\NN^{k}}$
satisfies $a_{\sigma(\ii)}=a_{\ii}$ for every $\ii\in\NN^{k}$ and
$\sigma\in S_{\infty}$, then the set $\{\AbsoluteValue{a_{\ii}}\mid\ii\in\NN^{k}\}$
has at most $B(k)$ elements, and is therefore bounded so that
\[
\Norm v^{p}=\sum_{\ii\in\NN^{k}}\AbsoluteValue{a_{\ii}}^{p}\mu_{\ii}^{p}\leq\max_{\ii\in\NN^{k}}\AbsoluteValue{a_{\ii}}^{p}\sum_{(i_{1},\dots,i_{k})\in\NN^{k}}\Biggl(\prod_{\ell=1}^{k}\mu_{i_{\ell}}^{p}\Biggr)=\max_{\ii\in\NN^{k}}\AbsoluteValue{a_{\ii}}^{p}\Biggl(\sum_{i=1}^{\infty}\mu_{i}^{p}\Biggr)^{k}<\infty.
\]
Thus, $(\overline{V^{\otimes k}})^{S_{\infty}}$ is $B(k)$-dimensional.
Since all $m_{\pi}$ are $S_{\infty}$-invariant, they are elements
of $\overline{V^{\otimes k}}$. The set $\{m_{\pi}\mid\pi\textnormal{ a set partition of }[k]\}$
is now easily seen to be a linearly independent subset of $(\overline{V^{\otimes k}})^{S_{\infty}}$
with $B(k)$ elements, which completes the proof.\end{proof}

\begin{thm}\label{thm:Centralizer_in_Lp_case} The centralizer of
$\CC S_{\infty}$ in $\mathcal{B}(\overline{V^{\otimes k}})$ is isomorphic
to the finite-dimensional algebra $U_{k}$ of uniform block permutations.\end{thm}

\begin{proof} Recall that every $A\in\mathcal{B}(\overline{V^{\otimes k}})$
is determined by its images on $\{v_{\ii}\}_{\ii\in\NN^{k}}$, which
we arrange in a matrix $(A_{\ii}^{\jj})_{\ii,\jj\in\NN^{k}}\in\CC^{\NN^{k}\times\NN^{k}}$
such that for each $\ii\in\NN^{k}$, 
\[
A(v_{\ii})=\sum_{\jj\in\NN^{k}}A_{\ii}^{\jj}v_{\jj}\in\overline{V^{\otimes k}}.
\]
Similarly to~\eqref{eq:finite_commutation}, if $\sigma\in S_{\infty}$
and $A\sigma=\sigma A$ as elements in $\mathcal{B}(\overline{V^{\otimes k}})$,
then for each $\ii\in\NN^{k}$, 
\[
\sigma A(v_{\ii})=\sum_{\jj\in\NN^{k}}A_{\ii}^{\jj}v_{\sigma(\jj)}\qquad\textnormal{equals}\qquad A\sigma(v_{\ii})=\sum_{\jj\in\NN^{k}}A_{\sigma(\ii)}^{\sigma(\jj)}v_{\sigma(\jj)}.
\]
Hence, $A$ is in the centralizer of $\CC S_{\infty}$ if and only
if 
\begin{equation}
A_{\sigma(\ii)}^{\sigma(\jj)}=A_{\ii}^{\jj}\quad\textnormal{for every }\ii,\jj\in\NN^{k}\textnormal{ and }\sigma\in\mathbb{S_{\infty}}.\label{eq:Condition_for_commuting_with_S_infty}
\end{equation}
Analogous to the proof of Proposition~\ref{prop:S_infty_invariant_subspace},
one finds that the diagonal action of $S_{\infty}$ on $\NN^{k}\times\NN^{k}$,
given by $\sigma(\ii,\jj)=(\sigma(\ii),\sigma(\jj))$ for $\ii,\jj\in\NN^{k}$,
has a finite number of orbits indexed by set partitions of $[k]\cup[k']$.
In particular, $(A_{\ii}^{\jj})_{\ii,\jj\in\NN^{k}}$ has uniform
entries on $S_{\infty}$-orbits if and only if it is a linear combination
of the finitely many diagram matrices defined by~\eqref{eq:diagram_matrix_entries}.
It remains to show that the combinations that yield bounded operators
correspond to elements of the algebra $U_{k}$. With this in mind,
let $A$ be in the centralizer of $\CC S_{\infty}$ in $\mathcal{B}(\overline{V^{\otimes k}})$.
For each $\ii_{0},\jj_{0}\in\NN^{k}$ and $\sigma\in S_{\infty}$,
we have 
\[
\Norm{A(v_{\sigma(\ii_{0})})}=\Norm{A\sigma(v_{\ii_{0}})}=\Norm{\sigma A(v_{\ii_{0}})}=\NormBig{\sum_{\jj\in\NN^{k}}A_{\ii_{0}}^{\jj}v_{\sigma(\jj)}}\geq\Norm{A_{\ii_{0}}^{\jj_{0}}v_{\sigma(\jj_{0})}}=\AbsoluteValue{A_{\ii_{0}}^{\jj_{0}}}\mu_{\sigma(\jj_{0})}.
\]
In particular, 
\[
\Norm A\geq\frac{\Norm{A(v_{\sigma(\ii_{0})})}}{\Norm{v_{\sigma(\ii_{0})}}}\geq\AbsoluteValue{A_{\ii_{0}}^{\jj_{0}}}\frac{\mu_{\sigma(\jj_{0})}}{\mu_{\sigma(\ii_{0})}}.
\]
If now some $\ell\in[k]$ appears $L>0$ times more often in $\ii_{0}$
than in $\jj_{0}$, and if $\sigma\in S_{\infty}$ fixes all entries
in $\ii_{0}$ and $\jj_{0}$ different from $\ell$, then 
\[
\frac{\mu_{\sigma(\jj_{0})}}{\mu_{\sigma(\ii_{0})}}=\left(\frac{\mu_{\ell}}{\mu_{\sigma(\ell)}}\right)^{L}\frac{\mu_{\jj_{0}}}{\mu_{\ii_{0}}}.
\]
Since $(\mu_{i})_{i\in\NN}\in\ell^{p}$, we have $\mu_{i}\to0$ for
$i\to\infty$ so that this quotient can be made arbitrarily large.
Hence, $A_{\ii_{0}}^{\jj_{0}}=0$ whenever $\ii_{0}$ and $\jj_{0}$
are not rearrangements of each other. On the other hand, every matrix
$(A_{\ii}^{\jj})_{\ii,\jj\in\NN^{k}}\in\CC^{\NN^{k}\times\NN^{k}}$
with this property that also satisfies~\eqref{eq:Condition_for_commuting_with_S_infty}
is a finite linear combination of diagram matrices $(d_{\ii}^{\jj})_{\ii,\jj\in\NN^{k}}$
as in~\eqref{eq:diagram_matrix_entries} that correspond to elements
of $U_{k}$. Note that for every such $(d_{\ii}^{\jj})_{\ii,\jj\in\NN^{k}}$
and $\ii_{0}\in\NN^{k}$, the vector $d(v_{\ii_{0}})=\sum_{\jj\in\NN^{k}}d_{\ii_{0}}^{\jj}v_{\jj}$
is either $0$ or equal to $v_{\jj_{0}}$ for some rearrangement $\jj_{0}$
of $\ii_{0}$, and $d(v_{\ii_{0}})=d(v_{\ii_{1}})\neq0$ implies $\ii_{0}=\ii_{1}$.
In particular, 
\[
d\colon\overline{V^{\otimes k}}\to\overline{V^{\otimes k}}\qquad\textnormal{given by}\qquad d\Bigl(\sum_{\ii\in\NN^{k}}a_{\ii}v_{\ii}\Bigr)=\sum_{\ii\in\NN^{k}}a_{\ii}d(v_{\ii})
\]
is a well-defined projector of norm $1$, which completes the proof.
\end{proof}

\begin{remark} If $p=2$, then the norm on $\overline{V^{\otimes k}}=L^{2}(\NN^{k},\mu^{k})$
originates from an inner product, and $\overline{V^{\otimes k}}$
is a Hilbert space. However, the action of $S_{\infty}$ would be
unitary only if $\mu^{k}$ was a multiple of the counting measure
on $\NN^{k}$, that is, if $\overline{V^{\otimes k}}\cong\ell^{2}$,
in which case $(\overline{V^{\otimes k}})^{S_{\infty}}$ would be
trivial. We refer to~\cite{Pickrell1988, Okounkov1997} and references
therein for the study of unitary representations of~$S_{\infty}$.\end{remark}

\begin{remark}\label{rk:S_infty_not_fully_decomposable} The group algebra $\CC S_{\infty}$
does not satisfy the double commutant property in $\cB(\overline{V^{\otimes k}})$,
that is, the centralizer of $U_{k}$ in $\cB(\overline{V^{\otimes k}})$
strictly contains $\CC S_{\infty}$. For example, when $k=1$, the
centralizer of $\CC S_{\infty}$ in $\cB(\overline{V^{\otimes1}})=\cB(V)$
is $U_{1}=\CC\{\textnormal{id}_{V}\}$; but the centralizer of $\CC\{\textnormal{id}_{V}\}$
is all of $\cB(V)$. This results from the fact that the action of
$S_{\infty}$ on $\cB(V)$ is not semisimple. In fact, the $S_{\infty}$-invariant
subspace $\CC\{\sum_{i\in\NN}v_{i}\}$ does not have closed $\CC S_{\infty}$-invariant
complements so that all projection operators with range $\CC\{\sum_{i\in\NN}v_{i}\}$
are unbounded. In particular, $V=\overline{V^{\otimes1}}$ is not
fully decomposable as a module for $\CC S_{\infty}$. A similar argument
applies for $k>1$.\end{remark}

\begin{remark}\label{rk:Uk_as_module_for_S_infty} Despite the delicate
properties of the action of $S_{\infty}$ on $\overline{V^{\otimes k}}$
mentioned in the previous remark, Proposition~\ref{prop:S_infty_invariant_subspace}
frames the vector space $(\overline{V^{\otimes k}})^{S_{\infty}}$
of $S_{\infty}$-invariant vectors as a natural module for the algebra
$U_{k}$ of uniform block permutations. In fact, in~\cite{AguiarOrellana2008},
Aguiar and Orellana study the combinatorial Hopf algebra of uniform
block permutations, and find that the ring of symmetric functions
in non-commuting variables naturally lives in their algebra. A priori,
this may be surprising since in their setting, $U_{k}$ arises as
the centralizer of the seemingly unrelated complex reflection group
$C_{r}\wr S_{n}$ on a permutation-like module, as shown in\emph{~}\cite{Tanabe1997}.
However, there is a subtlety in the centralizer relationship depending
on the values of $k$ and $r$. One can use the action of $C_{r}\wr S_{n}$
on $\CC^{n}$ as defined in\emph{~}\cite[Section 3.1]{AguiarOrellana2008},
calculate the corresponding commutation conditions as in \eqref{eq:finite_commutation},
and take the limit as $r,n\to\infty$ to obtain the same commutation
conditions as in the proof of Theorem~\ref{thm:Centralizer_in_Lp_case}.
In light of this observation, the results of Aguiar and Orellana connecting
the Hopf algebra of uniform block permutations and the ring of symmetric
functions in non-commuting variables appear natural.\end{remark}

\subsection{Bounded sequences \label{subsec:l_infinity}}

In the following, we consider the Banach space of bounded sequences
\[
\ell^{\infty}=\biggl\{ v=(a_{1},a_{2},\ldots)\in\CC^{\NN}\,\biggl|\,\Norm v_{\infty}=\sup_{i\in\NN}\AbsoluteValue{a_{i}}<\infty\biggr\}.
\]
Recall that $\ell^{\infty}$ is not \emph{separable}, meaning that
it has no countable dense subsets. In particular, $\ell^{\infty}$
has no countable Schauder bases. For example, if $\{v_{i}\}_{i\in\NN}$
is given by $v_{i}=(\delta_{ij})_{j\in\NN}$, then $\{\sum_{i=1}^{\ell}v_{i}\}_{\ell\in\NN}$
is not a Cauchy sequence even though $(1,1,1,\ldots)\in\ell^{\infty}$.
Therefore, we no longer use the notation $\sum_{i\in\NN}a_{i}v_{i}$
for $(a_{1},a_{2},\ldots)\in\CC^{\NN}$. Following the previous section,
we regard $(\ell^{\infty})^{\otimes k}$ as a subspace of 
\[
\ell^{\infty}(\NN^{k})=\biggl\{ v=(a_{\ii})_{\ii\in\mathbb{\NN}^{k}}\in\CC^{\NN^{k}}\,\biggl|\,\Norm v_{\infty}=\sup_{\ii\in\NN^{k}}\AbsoluteValue{a_{\ii}}<\infty\biggr\}
\]
such that 
\[
v_{\ii}=(\delta_{\ii\jj})_{\jj\in\NN^{k}}=v_{i_{1}}\otimes\cdots\otimes v_{i_{k}},\textnormal{ where }\delta_{\ii\jj}=\begin{cases}
1 & \textnormal{if }\ii=\jj\\
0 & \textnormal{otherwise.}
\end{cases}
\]
Note that $(\ell^{\infty})^{\otimes k}$ is not dense in $\ell^{\infty}(\NN^{k})$
if $k>1$. Again, $\{v_{\ii}\}_{\ii\in\NN^{k}}$ is not a Schauder
basis. Thus, we only consider operators $A$ on $\ell^{\infty}(\NN^{k})$
which are determined by their associated matrix with entries $A_{\ii}^{\jj}=(A(v_{\ii}))_{\jj}$
for $\ii,\jj\in\NN^{k}$. To this end, let 
\[
\BoundedOperators_{\textnormal{Mat}}(\ell^{\infty}(\NN^{k}))=\biggl\{ A=(A_{\ii}^{\jj})_{\ii,\jj\in\NN^{k}}\in\CC^{\NN^{k}\times\NN^{k}}\,\biggl|\,\Norm A_{\textnormal{Mat}}=\sup_{\jj\in\NN^{k}}\Bigl\{\sum_{\ii\in\NN^{k}}\AbsoluteValue{A_{\ii}^{\jj}}\Bigr\}<\infty\biggr\}.
\]
If $A=(A_{\ii}^{\jj})_{\ii,\jj\in\NN^{k}}\in\BoundedOperators_{\textnormal{Mat}}(\ell^{\infty}(\NN^{k}))$
and $v=(b_{\ii})_{\ii\in\NN^{k}}\in\ell^{\infty}(\NN^{k})$, we define
for each $\jj\in\NN^{k}$, 
\[
(A(v))_{\jj}=\sum_{\ii\in\NN^{k}}A_{\ii}^{\jj}b_{\ii},
\]
which converges absolutely since 
\[
\sum_{\ii\in\NN^{k}}\AbsoluteValue{A_{\ii}^{\jj}}\AbsoluteValue{b_{\ii}}\leq\sum_{\ii\in\NN^{k}}\AbsoluteValue{A_{\ii}^{\jj}}\Norm v_{\infty}\leq\Norm A_{\textnormal{Mat}}\Norm v_{\infty}.
\]
Hence, $A(v)\in\ell^{\infty}(\NN^{k})$, and $A$ gives rise to an
operator on $\ell^{\infty}(\NN^{k})$ with norm bounded by $\Norm A_{\textnormal{Mat}}$.

\begin{lem}[{\cite[Sec.\ 13 Satz 4]{KoetheToeplitz1934}}] The norm $\Norm A_{\infty}$ of each $A\in\BoundedOperators_{\textnormal{Mat}}(\ell^{\infty}(\NN^{k}))$
as an operator on $\ell^{\infty}(\NN^{k})$ equals $\Norm A_{\textnormal{Mat}}$,
and $\BoundedOperators_{\textnormal{Mat}}(\ell^{\infty}(\NN^{k}))$
is a subalgebra of $\BoundedOperators(\ell^{\infty}(\NN^{k}))$.\end{lem}

\begin{proof} In order to show $\Norm A_{\infty}\geq\Norm A_{\textnormal{Mat}}$,
note that for each $\jj\in\NN^{k}$ we can find $w=(b_{\ii})_{\ii\in\NN^{k}}$
with $\Norm w_{\infty}=1$ and 
\[
(A(w))_{\jj}=\sum_{\ii\in\NN^{k}}A_{\ii}^{\jj}b_{\ii}=\sum_{\ii\in\NN^{k}}\AbsoluteValue{A_{\ii}^{\jj}}.
\]
Hence, if $\{\jj_{\ell}\}_{\ell\in\NN}$ is a sequence in $\NN^{k}$
such that $\sum_{\ii\in\NN^{k}}\AbsoluteValue{A_{\ii}^{\jj_{\ell}}}\to\Norm A_{\textnormal{Mat}}$,
then we can find a sequence of unit vectors $\{w_{\ell}\}_{\ell\in\NN}$
in $\ell^{\infty}(\NN^{k})$ such that $\sup_{\ell\in\NN}\Norm{A(w_{\ell})}_{\infty}\geq\Norm A_{\textnormal{Mat}}$,
showing that $\Norm A_{\infty}\geq\Norm A_{\textnormal{Mat}}$. As
for the second statement, $\BoundedOperators_{\textnormal{Mat}}(\ell^{\infty}(\NN^{k}))$
is easily seen to be a vector space, and it suffices to show that
it is closed under composition. If $A=(A_{\ii}^{\jj})_{\ii,\jj\in\NN^{k}},B=(B_{\ii}^{\jj})_{\ii,\jj\in\NN^{k}}\in\BoundedOperators_{\textnormal{Mat}}(\ell^{\infty}(\NN^{k}))$,
let $C=(C_{\ii}^{\jj})_{\ii,\jj\in\NN^{k}}$ have entries 
\[
C_{\ii}^{\kk}=\sum_{\jj\in\NN^{k}}A_{\jj}^{\kk}B_{\ii}^{\jj},
\]
which converge absolutely since 
\[
\sum_{\jj\in\NN^{k}}\AbsoluteValue{A_{\jj}^{\kk}}\AbsoluteValue{B_{\ii}^{\jj}}\leq\sum_{\jj\in\NN^{k}}\AbsoluteValue{A_{\jj}^{\kk}}\Norm B\leq\Norm A\Norm B.
\]
Moreover, for each $\kk\in\NN^{k}$, 
\[
\sum_{\ii\in\NN^{k}}\AbsoluteValue{C_{\ii}^{\kk}}\leq\sum_{\ii\in\NN^{k}}\sum_{\jj\in\NN^{k}}\AbsoluteValue{A_{\jj}^{\kk}}\AbsoluteValue{B_{\ii}^{\jj}}=\sum_{\jj\in\NN^{k}}\AbsoluteValue{A_{\jj}^{\kk}}\sum_{\ii\in\NN^{k}}\AbsoluteValue{B_{\ii}^{\jj}}\leq\sum_{\jj\in\NN^{k}}\AbsoluteValue{A_{\jj}^{\kk}}\Norm B\leq\Norm A\Norm B.
\]
Hence, $C\in\BoundedOperators_{\textnormal{Mat}}(\ell^{\infty}(\NN^{k}))$.
Unsurprisingly, $C$ equals the composition of $A$ and $B$ since
for every $v=(c_{\jj})_{\jj\in\NN^{k}}\in\ell^{\infty}(\NN^{k})$
and $\kk\in\NN^{k}$, 
\[
(C(v))_{\kk}=\sum_{\ii\in\NN^{k}}C_{\ii}^{\kk}c_{\ii}=\sum_{\ii\in\NN^{k}}\sum_{\jj\in\NN^{k}}A_{\jj}^{\kk}B_{\ii}^{\jj}c_{\ii}=\sum_{\jj\in\NN^{k}}A_{\jj}^{\kk}(B(v))_{\jj}=(A(B(v)))_{\kk}.
\]

\end{proof}

We define a norm-preserving action of $S_{\infty}$ on $\ell^{\infty}(\NN^{k})$
by 
\[
\sigma\cdot(a_{\ii})_{\ii\in\NN^{k}}=(a_{\sigma^{-1}(\ii)})_{\ii\in\NN^{k}}\qquad\textnormal{for }\sigma\in S_{\infty},
\]
so that for every $\ii\in\NN^{k}$, 
\[
\sigma\cdot v_{\ii}=v_{\sigma(\ii)}.
\]

\begin{thm}\label{thm:Centralizer_in_l_infty} The vector space $(\ell^{\infty}(\NN^{k}))^{S_{\infty}}$
of $S_{\infty}$-invariant elements in $\ell^{\infty}(\NN^{k})$ has
the functions $\{m_{\pi}\mid\pi\textnormal{ a set partition of }[k]\}$
defined by~\eqref{eq:Monomial_basis_for_S_infty_invariant_subspace}
as a basis. Moreover, the centralizer of $\CC S_{\infty}$ in $\BoundedOperators_{\textnormal{Mat}}(\ell^{\infty}(\NN^{k}))$
is isomorphic to the finite-dimensional bottom-propagating partition
algebra $BP_{k}$.\end{thm}

\begin{proof} The first statement is proven in essentially the same
way as Proposition~\ref{prop:S_infty_invariant_subspace}. As for
the second statement, we follow the proof of Theorem~\ref{thm:Centralizer_in_Lp_case},
and note that $A=(A_{\ii}^{\jj})_{\ii,\jj\in\NN^{k}}\in\BoundedOperators_{\textnormal{Mat}}(\ell^{\infty}(\NN^{k}))$
is in the centralizer of $\CC S_{\infty}$ if and only if it is a
linear combination of the finitely many diagram matrices defined by~\eqref{eq:diagram_matrix_entries}.
The claim now follows from the observation that if $A=(A_{\ii}^{\jj})_{\ii,\jj\in\NN^{k}}$
is such a linear combination, then $\{\ii\in\NN^{k}\mid A_{\ii}^{\jj}\neq0\}$
is finite for every $\jj\in\NN^{k}$ if and only if $A$ is a linear
combination of matrices corresponding to diagrams with no blocks isolated
to the bottom row, that is, diagrams in $BP_{k}$. \end{proof}

\begin{remark} The case surrounding the countable-dimensional vector space in Section \ref{sec:SamSnowden} can be reconsidered in a similar fashion. Namely, let $V=c_{00}\subset\ell^{\infty}$
be the set of all sequences $(a_{1},a_{2},\ldots)\in\CC^{\NN}$ whose
support $\{i\in\NN\mid a_{i}\neq0\}$ is finite. Then, $V^{\otimes k}$
can be regarded as a subspace of $\ell^{\infty}(\NN^{k})$, and it
has $\{v_{\ii}\}_{\ii\in\NN^{k}}$ as a countable Hamel basis. Thus,
every $A\in\End(V^{\otimes k})$ is uniquely determined by $\{A(v_{\ii})\}_{\ii\in\NN^{k}}$,
whose elements are arranged as a matrix $(A_{\ii}^{\jj})_{\ii,\jj\in\NN^{k}}\in\CC^{\NN^{k}\times\NN^{k}}$
so that for $\ii,\jj\in\NN^{k}$ and $v=(b_{\ii})_{\ii\in\NN^{k}}\in V^{\otimes k}$,
\[
A(v_{\ii})=\sum_{\jj\in\NN^{k}}A_{\ii}^{\jj}v_{\jj},\qquad\textnormal{giving}\qquad(A(v))_{\jj}=\sum_{\ii\in\NN^{k}}A_{\ii}^{\jj}b_{\ii}.
\]
In particular, a diagram matrix $(d_{\ii}^{\jj})_{\ii,\jj\in\NN^{k}}$
as in~\eqref{eq:diagram_matrix_entries} corresponds to an element
of $\End(V^{\otimes k})$ if and only if $\{\jj\in\NN^{k}\mid d_{\ii}^{\jj}\neq0\}$
is finite for every $\ii\in\NN^{k}$. This happens precisely if the
corresponding diagram has no blocks isolated to its top row, for which
reason $\mbox{End}_{S_{\infty}}(V^{\otimes k})\cong TP_{k}$.\end{remark}

\begin{remark} As in Section~\ref{subsec:p_Power_summable_sequences},
$S_{\infty}$ is strictly contained in its double commutant; see Remark~\ref{rk:S_infty_not_fully_decomposable}.
Similarly to Remark~\ref{rk:Uk_as_module_for_S_infty}, the vector
space $(\ell^{\infty}(\NN^{k}))^{S_{\infty}}$ becomes a natural module
for the bottom-propagating partition algebra $BP_{k}$,
and it has a basis consisting of elements which may be identified
with monomial symmetric functions. Of course, the entire partition
algebra $P_{k}(x)$ has a natural action on the set of set partitions
of $[k]$ obtained by identifying each set partition $\pi$ with the
diagram $d=\pi\cup\{\{1'\},\dots,\{k'\}\}$; these diagrams form a
left ideal in $P_{k}(x)$. One might expect an action of $P_{k}(x)$
on the basis $\{m_{\pi}\mid\pi\textnormal{ a set partition of }[k]\}$
of $(\ell^{\infty}(\NN^{k}))^{S_{\infty}}$. Recall that in Theorem~\ref{thm:PkFullCentralizer},
$x$ must be specialized to the number $n$ of basis vectors on
which $S_{n}$ acts. The transition $x\to\infty$ can be carried out
rigorously only if middle components are avoided in diagram concatenations,
in this case by restricting to $BP_{k}$.\end{remark}

\bibliographystyle{amsalpha}
\bibliography{S_infty}

\end{document}